\documentclass[a4paper,12pt,twoside]{article}

\usepackage{amsmath,amsthm,amscd,amsfonts,amssymb,amsxtra}
\usepackage[hmargin=1.2in,vmargin=1.2in]{geometry}
\usepackage[utf8]{inputenc}
\usepackage{graphicx}
\usepackage[auth-sc-lg,affil-it]{authblk}
\usepackage[pdftex,bookmarks,colorlinks=false]{hyperref}
\usepackage{enumitem}
\usepackage{caption,subcaption}
\usepackage[mathscr]{euscript}

\newcommand{\E}{\mu}
\newcommand{\C}{\mathcal{C}}
\newcommand{\V}{\sigma^2}
\newcommand{\cE}{\E_{\chi}}
\newcommand{\cV}{\V_{\chi}}
\newcommand{\dE}{\E_{\chi^+}}
\newcommand{\dV}{\V_{\chi^+}}

\newtheorem{thm}{Theorem}[section]

\newtheorem{cor}[thm]{Corollary}

\newtheorem{defn}[thm]{Definition}

\newtheorem{prop}[thm]{Proposition}

\newtheorem{rem}{Remark}[section]

\def\ni{\noindent}

\title{\sc On Certain Colouring Parameters of Graphs}

\author{N. K. Sudev\footnote{Corresponding Author}, S. Satheesh$^{\ast\ast}$}
\affil{\small Centre for Studies in Discrete Mathematics\\ Vidya Academy of Science \& Technology \\ Thrissur - 680501, Kerala, India.\\ E-mail: $^\ast$sudevnk@gmail.com, $^{\ast\ast}$ssatheesh1963@yahoo.co.in}

\author{K. P. Chithra}
\affil{\small Naduvath Mana, Nandikkara\\ Thrissur-680301, Kerala, India.\\ E-mail: chithrasudev@gmail.com}

\author{Johan Kok}
\affil{\small Tshwane Metropolitan Police Department\\ City of Tshwane, South africa \\ E-mail: kokkiek2@tshwane.gov.za}

\date{}

\begin{document}
\maketitle

\begin{abstract}
Colouring the vertices of a graph $G$ according to certain conditions can be considered as a random experiment and a discrete random variable $X$ can be defined as the number of vertices having a particular colour in the proper colouring of $G$. In this paper, we extend the concepts of mean and variance, two important statistical measures, to the theory of graph colouring and determine the values of these parameters for a number of standard graphs.
\end{abstract}

\ni {\small \bf Key Words}: Graph Colouring; colouring sum of graphs; colouring mean; colouring variance; $\chi$-chromatic; $\chi^+$-chromatic.  

\vspace{0.2cm}

\ni {\small \bf Mathematics Subject Classification}: 05C15, 62A01.

\section{Introduction}

Investigations on graph colouring problems have attracted wide interest among researchers since its introduction in the second half of the nineteenth century. The vertex colouring or simply a colouring of a graph is an assignment of colours or labels to the vertices of a graph subject to certain conditions. In a proper colouring of a graph, its vertices are coloured in such a way that no two adjacent vertices in that graph have the same colour. 

Different types of graph colourings have been introduced during several subsequent studies. Many practical and real life situations paved paths to different graph colouring problems.  

Several researchers have also introduced various parameters related to different types of graph colouring and studied their properties extensively. The first and the most important parameter in the theory of graph colouring is the \textit{chromatic number} of graphs which is defined as the minimum number of colours required in a proper colouring of the given graph. The concept of chromatic number has been extended to almost all types of graph colourings. 

The notion of chromatic sums of graphs related to various graph colourings have been introduced and studied extensively. Some of these studies can be found in \cite{KSC1,ES1,EK1}. The notion of a general colouring sum of a graph has been explained in \cite{KSC1} as follows. 

Let $\C = \{c_1,c_2, c_3, \ldots,c_k\}$ be a particular type of proper $k$-colouring of a given graph $G$ and $\theta(c_i)$ denotes the number of times a particular colour $c_i$ is assigned to vertices of $G$. Then, the \textit{colouring sum} of a colouring $\C$ of a given graph $G$, denoted by $\omega_{\C}(G)$, is defined to be $\omega_{\C}(G) = \sum\limits_{i=1}^{k}i\,\theta(c_i)$.

Motivated by the studies on different types of graph colouring problems, corresponding parameters and their applications, we discuss the concepts of mean and variance, two important statistical parameters, to the theory of graph colouring in this paper.

For all  terms and definitions, not defined specifically in this paper, we refer to \cite{BM1,BLS,CZ, FH,EWW,DBW} and for the terminology of graph colouring, we refer to \cite{CZ1,JT1,MK1}.  For the concepts in Statistics, please see \cite{VS1,SMR1}. Unless mentioned otherwise, all graphs considered in this paper are simple, finite, connected and non-trivial.

\section{Colouring Mean and Variance of Graphs}

We can identify the colouring of the vertices of a given graph $G$ with a random experiment. Let $\C = \{c_1,c_2, c_3, \ldots,c_k\}$ be a proper $k$-colouring of $G$ and let $X$ be the random variable (\textit{r.v}) which denotes the number of vertices in $G$ having a particular colour. Since the sum of all weights of colours of $G$ is the order of $G$, the real valued function $f(i)$ defined by 
$$f(i)= 
\begin{cases}
\frac{\theta(c_i)}{|V(G)|}; &  i=1,2,3,\ldots,k\\
0; & \text{elsewhere}
\end{cases}$$ 
is the probability mass function (\textit{p.m.f})of the \textit{r.v} $X$. If the context is clear, we can also say that $f(i)$ is the \textit{p.m.f} of the graph $G$ with respect to the given colouring $\C$.

Hence, analogous to the definitions of the mean and variance of random variables, the mean and variance of a graph $G$, with respect to a general colouring of $G$ can be defined as follows.

\begin{defn}\label{Defn-1.1}{\rm 
Let $\C = \{c_1,c_2, c_3, \ldots,c_k\}$ be a certain type of proper $k$-colouring of a given graph $G$ and $\theta(c_i)$ denotes the number of times a particular colour $c_i$ is assigned to vertices of $G$. Then, the \textit{colouring mean} of a colouring $\C$ of a given graph $G$, denoted by $\E_{\C}(G)$, is defined to be $\E_{\C}(G) = \frac{\sum\limits_{i=1}^{k}i\,\theta(c_i)}{\sum\limits_{i=1}^k\theta(c_i)}$.}
\end{defn}

\begin{defn}\label{Defn-1.2}{\rm 
For a positive integer $r$, the $r$-th moment of the colouring $\C$ is denoted by $\E_{\C^r}(G)$ and is defined as $\E_{\C^r}(G) = \frac{\sum\limits_{i=1}^{k}i^r\,\theta(c_i)}{\sum\limits_{i=1}^k\theta(c_i)}$.}
\end{defn}

\begin{defn}\label{Defn-1.3}{\rm 
The \textit{colouring variance} of a colouring $\C$ of a given graph $G$, denoted by $\V_{\C}(G)$, is defined to be $\V_{\C}(G) = \frac{\sum\limits_{i=1}^{k}i^2\,\theta(c_i)}{\sum\limits_{i=1}^k\theta(c_i)}-\left(\frac{\sum\limits_{i=1}^{k}i^2\,\theta(c_i)}{\sum\limits_{i=1}^k\theta(c_i)}\right)^2$.}
\end{defn}

\subsection{$\chi$-Chromatic Mean and Variance of Graphs}

Colouring mean and variance corresponding to a particular type of minimal proper colouring of the vertices of $G$ are defined as follows.

\begin{defn}{\rm 
A colouring mean of a graph $G$, with respect to a proper colouring $\C$ is said to be a \textit{$\chi$-chromatic mean} of $G$, if $\C$ is the minimum proper colouring of $G$ and the colouring sum $\omega_G$ is also minimum. The $\chi$-chromatic mean of a graph $G$ is denoted by $\cE(G)$. }
\end{defn}

\begin{defn}{\rm 
The \textit{$\chi$-chromatic variance} of $G$, denoted by $\cV(G)$, is a colouring variance of $G$ with respect to a minimal proper colouring $\C$ of $G$ which yields the minimum colouring sum.  }
\end{defn}

Let us now determine the $\chi$-chromatic mean and variance of certain standard graph classes. The following result discusses the $\chi$-chromatic mean and variance of a complete graph $K_n$.

\begin{prop}\label{Prop-1.2}
The $\chi$-chromatic mean of a complete graph $K_n$ is $\frac{n+1}{2}$ and its $\chi$-chromatic variance is $\frac{n^2-1}{12}$.
\end{prop}
\begin{proof}
Note that all vertices of a complete graph $K_n$ must have different colours as they are all adjacent to each other. That is, $\theta(c_i)=1$ for colour $c_i, ~1\le i \le n$. Therefore, $\cE(K_n)=\frac{1}{n}\sum\limits_{i=1}^n\, i=\frac{n+1}{2}$ and $\cV(K_n)=\frac{1}{n}\sum\limits_{i=1}^n\,i^2 - \left(\frac{n+1}{2}\right)^2=\frac{(n+1)(2n+1)}{6}-\frac{(n+1)^2}{2}=\frac{n^2-1}{12}$.
\end{proof}

The following theorem gives the probability distribution of a proper colouring  of a complete graph.

\begin{thm}\label{Thm-1.1}
Any proper colouring of a complete graph $K_n$ has the discrete uniform distribution on \{1,2, \ldots,k\} (DU(k)).
\end{thm}
\begin{proof}
Let $X$ be the \textit{r.v} representing the number of colours in a proper $k$-colouring of a complete graph $K_n$. For any proper $k$-colouring $\C$ of the complete graph $K_n$, $\theta(c_i)=1$ and $k=n$. Hence, the corresponding \textit{p.m.f} is
$$f(i)=
\begin{cases}
\frac{1}{n}; & n=1,2,3,\ldots, n,\\
0; & \text{elsewhere}
\end{cases}$$
which is that of the discrete uniform distribution on \{1,2, \ldots, k\}. Hence, $X \sim DU(k)$.
\end{proof} 

\ni The following result determines the $\chi$-chromatic mean and variance for a path $P_n$.

\begin{prop}\label{Prop-1.3}
The $\chi$-chromatic mean of a path $P_n$ is  
\begin{equation*}
\cE(P_n)=
\begin{cases}
\frac{3}{2}; &  \text{if $n$  is even},\\
\frac{3n-1}{2n};  &  \text{if $n$  is odd},
\end{cases}
\end{equation*}
and the $\chi$-chromatic variance of $P_n$ is
\begin{equation*}
\cV(P_n)=
\begin{cases}
\frac{1}{4}; &  \text{if $n$  is even},\\
\frac{n^2-1}{4n^2};  &  \text{if $n$  is odd}.
\end{cases}
\end{equation*}  
\end{prop}
\begin{proof}
Consider a path $P_n$ on $n$ vertices. Being a bipartite graph, the vertices of $P_n$ can be coloured using two colours, say $c_1$ and $c_2$. Then, we have the following cases.

\begin{enumerate}
\item[(i)] If $n$ is even, exactly $\frac{n}{2}$ vertices of $P_n$ have colour $c_1$ and $\frac{n}{2}$ vertices have colour $c_2$. Then, the \textit{p.m.f} of the corresponding \textit{r.v} $X$ is 
$$f(i)=
\begin{cases}
\frac{1}{2}; & i=1,2, \\
0; & \text{elsewhere}.
\end{cases}$$

Hence, the $\chi$-chromatic mean is $\cE(P_n)=\sum\limits_{i=1}^{2}\, i\,\frac{1}{2}=\frac{3}{2}$ and the $\chi$-chromatic variance is $\cV(P_n)=\sum\limits_{i=1}^{2}\, i^2\,\frac{1}{2}-\left(\cE\right)^2=\frac{5}{2}-\left(\frac{3}{2}\right)^2=\frac{1}{4}$.

\item[(ii)] If $n$ is odd, then the \textit{p.m.f} of the corresponding \textit{r.v} $X$ is 
$$f(i)=
\begin{cases}
\frac{n+1}{2n}; & i=1, \\
\frac{n-1}{2n}; & i=2, \\
0; & \text{elsewhere}.
\end{cases}$$
Then, the $\chi$-chromatic mean of $P_n$ is $\cE(P_n)=1\cdot \frac{n+1}{2n}+2\cdot \frac{n-1}{2n}=\frac{3n-1}{2n}$ and its $\chi$-chromatic variance is $\cV(P_n)=1^2\cdot \frac{n+1}{2n}+2^2\cdot\frac{n-1}{2n}-\left(\frac{3n-1}{2n}\right)^2=\frac{n^2-1}{4n^2}$.
\end{enumerate}
\vspace{-0.85cm}
\end{proof}

\ni The following result determines the values of these parameters for a cycle $C_n$.

\begin{prop}\label{Prop-1.4}
The $\chi$-chromatic mean of a cycle $C_n$ is  
\begin{equation*}
\cE(C_n)=
\begin{cases}
\frac{3}{2}; &  \text{if $n$  is even},\\
\frac{3n+3}{2n};  &  \text{if $n$  is odd},
\end{cases}
\end{equation*}
and the $\chi$-chromatic variance of $C_n$ is
\begin{equation*}
\cV(C_n)=
\begin{cases}
\frac{1}{4}; &  \text{if $n$  is even},\\
\frac{n^2-8n+9}{4n^2};  &  \text{if $n$  is odd}.
\end{cases}
\end{equation*}  
\end{prop}
\begin{proof}
Consider a cycle $C_n$ on $n$ vertices.  Then, we have the following cases.

\begin{enumerate}
\item[(i)] If $n$ is even, then $C_n$ is bipartite and is $2$-colourable. Then, exactly $\frac{n}{2}$ vertices of $C_n$ also have colour $c_1$ and $c_2$ each. Then, as explained in the first part of previous theorem,  we have $\E_{\chi}(C_n)=\frac{3}{2}$ and $\V_{\chi}(C_n)=\frac{1}{4}$.

\item[(ii)] If $n$ is odd, then $C_n$ is $3$-colourable. Let $\C=\{c_1,c_2,c_3\}$ be the minimal proper colouring of $C_n$. Then, the \textit{p.m.f} of the \textit{r.v} $X$ is given by
$$f(i)=
\begin{cases}
\frac{n-1}{2n}; & \text{if}~~ i=1,2,\\
\frac{1}{n}; & \text{if}~~ i=3,\\
0; & \text{elsewhere}.
\end{cases}$$
Then, the $\chi$-chromatic mean of $G$ is $\cE(C_n)=1\cdot \frac{n-1}{2n}+2\cdot \frac{n-1}{2n}+3\cdot \frac{1}{n}=\frac{3n+3}{2n}$ and the $\chi$-chromatic variance of $C_n$ is $\cV(C_n)=\left(1^2\cdot \frac{n-1}{2n}+2^2\cdot \frac{n-1}{2n}+3^2\cdot \frac{1}{n}\right)-\left(\frac{3n+3}{2n}\right)^2=\frac{n^2-8n+9}{4n^2}$.
\end{enumerate}
\vspace{-0.75cm}
\end{proof}

In the following theorem, we determine the $\chi$-chromatic mean and variance of a wheel graph $W_{n}=K_1+C_{n-1}$.

\begin{prop}\label{Prop-1.6}
The $\chi$-chromatic mean of a wheel graph $W_n$ is  
\begin{equation*}
\cE(W_n)=
\begin{cases}
\frac{3n+3}{2n}; &  \text{if $n$  is odd},\\
\frac{3n+1}{2n+2};  &  \text{if $n$  is even},
\end{cases}
\end{equation*}
and the $\chi$-chromatic variance of $W_n$ is
\begin{equation*}
\cV(W_n)=
\begin{cases}
\frac{n^2+8n-9}{4n^2}; &  \text{if $n$  is odd},\\
\frac{n^2+32n-64}{4n^2};  &  \text{if $n$  is even}.
\end{cases}
\end{equation*}
\end{prop}
\begin{proof}
Note that the wheel graph $W_n$ is $3$-colourable, when $n$ is odd and $4$-colourable when is even. Then, we have the following cases.

\begin{enumerate}
\item[(i)] First, assume that $n$ is an odd integer. Then, the outer cycle $C_{n-1}$ of $W_n$ is an even cycle. Hence, $\frac{n-1}{2}$ vertices of $C_{n-1}$ have colour $c_1$, $\frac{n-1}{2}$ vertices of $C_{n-1}$  have colour $c_2$ and the central vertex of $W_n$  has colour $c_3$. Hence the corresponding \textit{p.m.f} for $W_n$ is given by
$$f(i)=
\begin{cases}
\frac{n-1}{2n}; & \text{if}~~~ i=1,2,\\
\frac{1}{n}; &  \text{if}~~~ i=3,\\
0; & \text{elsewhere}.
\end{cases}$$

Hence, the corresponding $\chi$-chromatic mean is $\cE(W_n)=1\cdot \frac{n-1}{2n}+2\cdot \frac{n-1}{2n}+3\cdot \frac{1}{n}=\frac{3n+3}{2n}$. Now, the $\chi$-chromatic variance is $\cV(W_n)= (1^2+2^2)\cdot \frac{n-1}{2n}+3^2\cdot \frac{1}{n}-\left(\cE(W_n)\right)^2=\left(\frac{5(n-1)}{2n}+\frac{9}{n}\right)-\left(\frac{3n+3}{2n}\right)^2=\frac{n^2+8n-9}{4n^2}$.

\item[(ii)] Next, assume that $n$ is an even integer. Then, the outer cycle $C_{n-1}$ of $W_n$ is an odd cycle. Hence, $\frac{n-2}{2}$ vertices of the outer cycle $C_{n-1}$ have colour $c_1$, $\frac{n-2}{2}$ vertices of $C_{n-1}$ have colour $c_2$ and one vertex of $C_{n-1}$ has colour $c_3$ and the central vertex of $W_n$ has the $c_4$. Hence, the \textit{p.m.f} for $W_n$ is given by
$$f(i)=
\begin{cases}
\frac{n-2}{2n}; & \text{if}~~~ i=1,2,\\
\frac{1}{n}; &  \text{if}~~~ i=3,4\\
0; & \text{elsewhere}.
\end{cases}$$

Hence, the corresponding $\chi$-chromatic mean is $\cE(W_n)=1\cdot \frac{n-2}{2n}+2\cdot \frac{n-2}{2n}+3\cdot \frac{1}{n}+4\cdot \frac{1}{n}=\frac{3n+8}{2n}$ and the $\chi$-chromatic variance is $\cV(W_n)= (1^2+2^2)\cdot \frac{n-2}{2n}+(3^2+4^2)\cdot \frac{1}{n}-\left(\cE(W_n)\right)^2=\left(\frac{5(n-2)}{2n}+\frac{3^2+4^2}{n}\right)-\left(\frac{3n+8}{2n}\right)^2=\frac{n^2+32n-64}{4n^2}$.
\end{enumerate}
\vspace{-0.95cm}
\end{proof}

\begin{rem}{\rm 
From the above discussion, we observe that the minimum proper colouring of bipartite graph follows a two-point distribution. In general, for a bipartite graph $G(V_1,V_2,E)$, with $|V_1|=m_1>|V_2|=m_2, m_1+m_2=n$, 
the \textit{p.m.f} can be defined as 
$$f(i)=
\begin{cases}
\frac{m_1}{n}; & \text{if}~~~ i=1,\\
\frac{m_2}{n}; & \text{if}~~~ i=2,\\
0; & \text{elsewhere}.
\end{cases}$$

Hence, we have $\cE(G)=\frac{m_1+2m_2}{n}=1+\frac{m_2}{n}$ and $\cV(G)=\frac{m_1+4m_2}{n}-\left(1+\frac{m_2}{n}\right)^2=\frac{1}{n^2}\left[(n-1)m_1+2(2n-1)m_2\right]$.}
\end{rem}

\begin{rem}{\rm 
If $G$ is a regular bipartite graph on $n$ vertices, then there will be $\frac{n}{2}$ vertices in each partition and hence with respect to a minimal proper colouring, exactly $\frac{n}{2}$ vertices having the colours $c_1$ and $c_2$ each. Hence the \textit{p.m.f} is 

$$f(i)=
\begin{cases}
\frac{1}{2}; & i=1,2,\\
0; & \text{elsewhere}.
\end{cases}$$
Hence, $\cE(G)=\frac{3}{2}$ and $\cV(G)=\frac{1}{4}$ as mentioned in Proposition \ref{Prop-1.4}.}
\end{rem}

\subsection{$\chi^+$-Chromatic Mean and Variance of Graphs}

Colouring mean and variance corresponding to another type of a minimal proper colouring of the vertices of $G$ are defined as follows.

\begin{defn}{\rm 
A colouring mean of a graph $G$, with respect to a proper colouring $\C$ is said to be a \textit{$\chi^+$-chromatic mean} of $G$, if $\C$ is a minimum proper colouring of $G$ such that the corresponding colouring sum $\omega_G$ is maximum. The $\chi^+$-chromatic number of a graph $G$ is denoted by $\dE(G)$. }
\end{defn}

\begin{defn}{\rm 
The \textit{$\chi^+$-chromatic variance} of $G$, denoted by $\cV(G)$, is a colouring variance of $G$ with respect to a minimal proper colouring $\C$ of $G$ such that the corresponding colouring sum is maximum.  }
\end{defn}

Invoking the definitions of two types of chromatic means and variances mentioned above, we can infer the following. 

\begin{rem}{\rm 
For any arbitrary minimal proper colouring $\C$ of a given graph $G$, we have $\cE(G) \le \E_{\C}(G) \le \dE(G)$ and $\cV(G) \le \V_{\C}(G) \le \dV(G)$.}
\end{rem}

\begin{rem}{\rm 
Since all vertices of a complete graph have different colours, the $\chi$-chromatic mean and the $\chi^+$-chromatic mean are equal and the $\chi$-chromatic variance and the $\chi^+$-chromatic variance are equal.}
\end{rem}

Let us now discuss the $\chi^+$-chromatic mean and variance of the graph classes mentioned in the previous section.


\begin{prop}\label{Prop-2.3}
The $\chi^+$-chromatic mean of a path $P_n$ is  
\begin{equation*}
\dE(P_n)=
\begin{cases}
\frac{3}{2}; &  \text{if $n$  is even},\\
\frac{3n-1}{2n};  &  \text{if $n$  is odd},
\end{cases}
\end{equation*}
and the $\chi^+$-chromatic variance of $P_n$ is
\begin{equation*}
\dV(P_n)=
\begin{cases}
\frac{1}{4}; &  \text{if $n$  is even},\\
\frac{n^2-1}{4n^2};  &  \text{if $n$  is odd}.
\end{cases}
\end{equation*}  
\end{prop}
\begin{proof}
\ni As in Proposition \ref{Prop-1.3}, we consider the following cases.

\begin{enumerate}
\item[(i)] If $n$ is even, as mentioned in Proposition \ref{Prop-1.3}, exactly $\frac{n}{2}$ vertices of $P_n$ have colour $c_1$ and $\frac{n}{2}$ vertices have colour $c_2$. Then, the \textit{p.m.f} of the corresponding \textit{r.v} $X$ is also as defined there. 
Hence, the $\chi^+$-chromatic mean is $\cE(P_n)=\frac{3}{2}$ and the $\chi$-chromatic variance is $\cV(P_n)=\frac{1}{4}$.

\item[(ii)] If $n$ is odd, $\chi^+$-colouring assigns colour $c_1$ to $\frac{n-1}{2n}$ vertices and colour $c_2$ to the remaining $\frac{n+1}{2n}$ vertices. Then the \textit{p.m.f} is 
$$f(i)=
\begin{cases}
\frac{n-1}{2n}; & i=1, \\
\frac{n+1}{2n}; & i=2, \\
0; & \text{elsewhere}.
\end{cases}$$
Then, the $\chi^+$-chromatic mean of $P_n$ is given by $\cE(P_n)=1\cdot \frac{n-1}{2n}+2\cdot \frac{n+1}{2n}=\frac{3n+1}{2n}$ and its $\chi^+$-chromatic variance is given by $\cV(P_n)=1^2\cdot \frac{n-1}{2n}+2^2\cdot\frac{n+1}{2n}-\left(\frac{3n+1}{2n}\right)^2=\frac{n^2+1}{4n^2}$.
\end{enumerate}
\vspace{-0.75cm}
\end{proof}

The following proposition discusses the $\chi^+$-chromatic mean and variance of a cycle on $n$ vertices.

\begin{prop}\label{Prop-2.4}
The $\chi^+$-chromatic mean of a cycle $C_n$ is  
\begin{equation*}
\dE(C_n)=
\begin{cases}
\frac{3}{2}; &  \text{if $n$  is even},\\
\frac{5n-3}{2n};  &  \text{if $n$  is odd},
\end{cases}
\end{equation*}
and the $\chi^+$-chromatic variance of $P_n$ is
\begin{equation*}
\dV(C_n)=
\begin{cases}
\frac{1}{4}; &  \text{if $n$  is even},\\
\frac{n^2+8n-9}{4n^2};  &  \text{if $n$  is odd}.
\end{cases}
\end{equation*}  
\end{prop}
\begin{proof}
\ni Here, we have to consider the following two cases.

\begin{enumerate}
\item[(i)] If $n$ is even, as mentioned in Proposition \ref{Prop-2.3}, exactly $\frac{n}{2}$ vertices of $C_n$ have colour $c_1$ and colour $c_2$ each. Then, exactly as explained there, we have, $\cE(C_n)=\frac{3}{2}$ and $\cV(C_n)=\frac{1}{4}$.

\item[(ii)] If $n$ is odd, $\chi^+$-colouring assigns colour $c_1$ to one vertex, colour $c_2$ to $\frac{n-1}{2n}$ vertices and colour $c_3$ to the remaining $\frac{n-1}{2n}$ vertices of the cycle $C_n$. Then the \textit{p.m.f} is 
$$f(i)=
\begin{cases}
1; & i=1,\\
\frac{n-1}{2n}; & i=2,3 \\
0; & \text{elsewhere}.
\end{cases}$$
Then, the $\chi^+$-chromatic mean of $C_n$ is $\cE(C_n)=1\cdot \frac{1}{2n}+2\cdot \frac{n-1}{2n}+3\cdot \frac{n-1}{2n}=\frac{5n-3}{2n}$ and its $\chi^+$-chromatic variance is $\cV(C_n)=1^2\cdot \frac{1}{n}+2^2\cdot\frac{n+1}{2n}+3^2\cdot \frac{n-1}{2n}-\left(\frac{5n-3}{2n}\right)^2=\frac{n^2+8n-9}{4n^2}$.
\end{enumerate}
\vspace{-0.65cm}
\end{proof}

The following proposition discusses the $\chi^+$-chromatic mean and variance of a wheel graph on $n$ vertices.

\begin{prop}\label{Prop-2.5}
The $\chi^+$-chromatic mean of a wheel graph $W_n$ is  
\begin{equation*}
\dE(W_n)=
\begin{cases}
\frac{5n-3}{2n}; &  \text{if $n$  is odd},\\
\frac{3n+1}{2n+2};  &  \text{if $n$  is even},
\end{cases}
\end{equation*}
and the $\chi^+$-chromatic variance of $W_n$ is
\begin{equation*}
\dV(W_n)=
\begin{cases}
\frac{n^2+30n-31}{4n^2}; &  \text{if $n$  is odd},\\
\frac{n^2+32n-64}{4n^2};  &  \text{if $n$  is even}.
\end{cases}
\end{equation*}
\end{prop}
\begin{proof}
As mentioned in Proposition \ref{Prop-1.6}, the wheel graph $W_n$ is $3$-colorable, when $n$ is odd and $4$-colourable when is even. Then, we have to consider the following cases.

\begin{enumerate}
\item[(i)] First, assume that $n$ is an odd integer. Then, the outer cycle $C_{n-1}$ of $W_n$ is an even cycle. Hence, we can assign colour $c_1$ to the central vertex of $W_n$, colour $c_2$ to $\frac{n-1}{2}$ vertices of $C_{n-1}$ and colour $c_3$ to the remaining $\frac{n-1}{2}$ vertices of $C_{n-1}$. Hence the corresponding \textit{p.m.f} for $W_n$ is given by
$$f(i)=
\begin{cases}
\frac{1}{n}; &  \text{if}~~~ i=1,\\
\frac{n-1}{2n}; & \text{if}~~~ i=2,3,\\
0; & \text{elsewhere}.
\end{cases}$$

Hence, the $\chi$-chromatic mean is $\cE(W_n)=1\cdot \frac{1}{n}+2\cdot \frac{n-1}{2n}+3\cdot \frac{n-1}{2n}=\frac{5n-3}{2n}$ and the $\chi$-chromatic variance is $\cV(W_n)= 1^2\cdot \frac{1}{n}+(2^2+3^2)\cdot \frac{n-1}{2n}-\left(\cE(W_n)\right)^2=\left(\frac{13(n-1)}{2n}+\frac{1}{n}\right)-\left(\frac{5n-3}{2n}\right)^2=\frac{n^2+30n-31}{4n^2}$.

\item[(ii)] Let $n$ be an even integer. Then, the outer cycle $C_{n-1}$ of $W_n$ is an odd cycle. Hence, we can assign colour $c_1$ to the central vertex of $W_n$, colour  $c_2$ to one vertex of the outer cycle $C_{n-1}$, colour $c_3$ to $\frac{n-2}{2}$ vertices of $C_{n-1}$ and colour $c_4$ to the remaining  $\frac{n-2}{2}$ vertices of $C_{n-1}$. Therefore, the corresponding \textit{p.m.f} for $W_n$ is given by
$$f(i)=
\begin{cases}
\frac{1}{n}; &  \text{if}~~~ i=1,2\\
\frac{n-2}{2n}; & \text{if}~~~ i=3,4,\\
0; & \text{elsewhere}.
\end{cases}$$

Hence, the corresponding $\chi$-chromatic mean is $\cE(W_n)=1\cdot \frac{1}{n}+2\cdot \frac{1}{n}+3\cdot \frac{n-2}{2n}+4\cdot \frac{n-2}{2n}=\frac{7n-8}{2n}$ and the $\chi$-chromatic variance is $\cV(W_n)= (1^2+2^2)\cdot \frac{1}{n}+(3^2+4^2)\cdot \frac{n-2}{2n}-\left(\cE(W_n)\right)^2=\left(5\cdot\frac{1}{n}+25\cdot\frac{n-2}{2n}\right)-\left(\frac{7n-8}{2n}\right)^2=\frac{n^2+32n-64}{4n^2}$.
\end{enumerate}
\vspace{-0.95cm}
\end{proof}

\subsection{Some Interpretations}

A \textit{block graph} or \textit{clique tree} $G$ is an undirected graph in which every biconnected component (block) is a clique. By Theorem \ref{Thm-1.1}, minimum proper colouring of every component of $G$ follows uniform distribution. Hence, we have 

\begin{thm}
The probability distribution of a block graph $G$ is mixture of discrete uniform distributions. 
\end{thm}

An $n$-partite graph is a graph whose set of vertices can be partitioned in to $n$ subsets such that no two vertices in  the same partitions are adjacent. Then, we have the following result.

\begin{thm}
Let $G$ be a regular $k$-partite graph on vertices. Then, any minimal proper colouring of $G$ follows uniform distribution (in each partition). 
\end{thm}
\begin{proof}
Any minimal proper colouring of a $k$-partite graph contains $k$-colours. Let $G$ be an $r$-regular $k$-partite graph. Then, $rk=n$. Then, the \textit{p.m.f} of $G$ is 
$$f(i)=
\begin{cases}
\frac{1}{k}; & i=1,2,3,\ldots,k,\\
0; & \text{elsewhere}.
\end{cases}$$
which is that of the DU(k) distribution. 
\end{proof}

\begin{cor}
Let $G$ be a $k$-partite graph. Then, the $\chi$-chromatic mean (and $\chi^+$-chromatic mean) of $G$ is $\frac{k+1}{2}$ and the $\chi$-chromatic variance (and $\chi^+$-chromatic variance) of $G$ is $\frac{k^2-1}{12}$.
\end{cor}
\begin{proof}
The proof follows immediately from the fact that the minimal proper colouring of a $k$-partite graph follows uniform distribution.
\end{proof}
 
Certain areas where these notions can be made use of are: nodes in communication and traffic networks.

 \section{Scope for Further Studies}

In this paper, we have extended the notions of mean and variance to the theory of graph colouring and determined their values for certain graphs and graph classes. More problems in this area are still open. 

The $\chi$-chromatic mean and variance of many other graph classes are yet to be studied. Determining the sum, mean and variance corresponding to the colouring of certain generalised graphs like generalised Petersen graphs, fullerence graphs etc. are some of the promising open problems. Studies on the sum, mean and variance corresponding to different types of edge colourings, map colourings, total colourings etc. of graphs also offer much for future studies.

We can associate many other parameters to graph colouring and other notions like covering, matching etc.  All these facts highlight a wide scope for future studies in this area.

\section*{Acknowledgement}

The first author of this article dedicates this paper to the memory Prof. (Dr.) D. Balakrishnan, Founder Academic Director, Vidya Academy of Science and Technology, Thrissur, India., who had been his mentor, the philosopher and the role model in teaching and research.

\end{document}